\documentclass[12pt]{article}
\usepackage{a4wide, amstext, amsmath, amssymb, graphicx, array, epsfig, psfrag}
\usepackage{amstext, amsmath, amssymb, graphicx, todonotes}
\usepackage{stmaryrd}
\usepackage[update,prepend]{epstopdf}

\newcommand{\bfI}{\boldsymbol I}

\newcommand{\mcK}{\mathcal{K}}
\newcommand{\mcE}{\mathcal{E}}

\newcommand{\mcF}{\mathcal{F}}
\newcommand{\mcA}{\mathcal{A}}
\newcommand{\mcT}{\mathcal{T}}

\newcommand{\tn}{|\mspace{-1mu}|\mspace{-1mu}|}

\newcommand{\Gammah}{{\Gamma_h}}
\newcommand{\nablas}{\nabla_\Gamma}
\newcommand{\nablash}{\nabla_{\Gamma_h}}
\newcommand{\nablashb}{\beta_h \cdot \nabla_{\Gamma_h}}

\newcommand{\divs}{\text{div}_\Gamma}
\newcommand{\divsh}{\text{div}_\Gammah}
\newcommand{\IR}{\mathbb{R}}

\newcommand{\Ps}{{P}_\Gamma}
\newcommand{\Psh}{{P}_\Gammah}

\newcommand{\mcKh}{{\mcK_h}}

\newcommand{\mcTh}{{\mcT_h}}
\newcommand{\mcEh}{{\mcE_h}}
\newcommand{\shone}{{s_{h,1}}}
\newcommand{\shtwo}{{s_{h,2}}}

\numberwithin{equation}{section}
\newtheorem{lem}{Lemma}[section]

\newtheorem{thm}{Theorem}[section]
\newtheorem{rem}{Remark}[section]

\newenvironment{proof}{\noindent \newline {\bf Proof.}}
{\hfill \mbox{\fbox{} } \newline}
\begin{document}

\title{\bf A Stabilized Cut Streamline Diffusion Finite Element 
Method for Convection-Diffusion Problems on Surfaces }
\author{Erik Burman
\footnote{Department of Mathematics, University College London,
  London, WC1E  6BT United Kingdom} 
\mbox{ }
 Peter Hansbo
\footnote{Department of Mechanical Engineering, J\"onk\"oping University,
SE-55111 J\"onk\"oping, Sweden.} 
\mbox{ }
Mats G.\ Larson
\footnote{Department of Mathematics and Mathematical Statistics, Ume{\aa} University, SE-90187 Ume{\aa}, Sweden} 
\mbox{ } 
\\
Andr\'e Massing
\footnote{Department of Mathematics and Mathematical Statistics, Ume{\aa} University, SE-90187 Ume{\aa}, Sweden} 
\mbox{ }
Sara Zahedi
\footnote{Department of Mathematics, KTH Royal Institute of Technology,
SE-100 44 Stockholm, Sweden}
}
\maketitle

\begin{abstract} We develop a stabilized cut finite element 
method for the stationary convection diffusion problem on a 
surface embedded in $\IR^d$. The cut finite element method 
is based on using an embedding of the surface into a three 
dimensional mesh consisting of tetrahedra and then using the 
restriction of the standard piecewise linear continuous elements 
to a piecewise linear approximation of the surface. The stabilization 
consists of a standard streamline diffusion stabilization 
term on the discrete surface and a so called normal gradient 
stabilization term on the full tetrahedral elements in the active 
mesh. We prove optimal order a priori error estimates in the 
standard norm associated with the streamline diffusion method 
and bounds for the condition number of the resulting stiffness 
matrix. The condition number is of optimal order $O(h^{-1})$ for 
a specific choice of method parameters.  Numerical example 
supporting our theoretical results are also included.
\end{abstract}
\maketitle

\section{Indroduction}

\paragraph{Contributions.}
We develop and analyze cut finite element method for the 
convection-diffusion problem on surfaces. The cut finite element 
method is constructed as follows: (i) The surface is embedded into 
a three dimensional domain equipped with a family of meshes. 
(ii) A piecewise linear approximation of the surface is computed for 
instance using an interpolation of the distance function. (iii) The active 
mesh is defined as the subset of elements that intersect the discrete 
surface. (iv) A finite element approximation is defined by using a 
variational formulation together with a restriction of the finite element 
space to the active mesh. 

In order to stabilize the method we add two stabilization terms: (i) A
streamline diffusion stabilization, which is added on the discrete surface. (ii) A normal gradient stabilization term 
on the full three dimensional tetrahedra in the active mesh, which 
provides control of the variation of the discrete solution in the 
direction normal to the surface. Streamline diffusion or Streamline 
upwind Petrov--Galerkin methods were introduced in  \cite{BrHu82} 
and \cite{JoNaPi84} for transport dominated problems in flat domains.

With these constructions we can show that the resulting discrete problem 
is well conditioned and that optimal order a priori error estimates in the standard streamline diffusion holds without relying on the presence of 
positive diffusion. In particular, the condition number is of optimal order $O(h^{-1})$ for a specific choice of method parameters. 

\paragraph{Earlier Work.} 
CutFEM, or trace FEM, for partial differential equations on surfaces 
was first introduced for the Laplace-Beltrami operator in
\cite{OlReGr09} without stabilization and is now a rapidly developing
technique. In \cite{BuHaLa15} a stabilized version based on so called
face stabilization or ghost stabilization, which provides control over
the jump in the normal gradient across interior faces in the active
mesh was introduced and analyzed. In particular it was shown that the
condition number scaled in an optimal way. In \cite{OlReXu14b} a
streamline diffusion trace finite element method for the
convection-diffusion problem on a surface was studied, here only
streamline diffusion stabilization was added and the antisymmetric
discretization of the convection term was used. In \cite{BuHaLaZa15}
the pure convection problem on a surface with face stabilization was
analyzed. In \cite{BuHaLaMaZa16} a full gradient stabilization method
for the Laplace-Beltrami operator was developed and analyzed. In
\cite{BuHaLaMa16} an abstract framework for analysis of cut finite
element methods on embedded manifolds of arbitrary codimension was
developed and, in particular, the normal gradient stabilization term
which we employ in this paper was introduced and analyzed (see also
\cite{GLR18} for an analysis in the case of high order approximation). Coupled bulk-surface problems were considered in 
\cite{BuHaLaZa16} and \cite{GrOlRe15}. Higher order versions of 
trace fem for the Laplace-Beltrami operator were analyzed in
\cite{Reu15, GLR18}. Finally in, \cite{HaLaZa15}, \cite{OlRe14}, and \cite{OlReXu14a}, extensions to time dependent problems were presented.


Several other techniques for solving partial differential equations on
surfaces have 
been proposed. Most notably, the original idea of Dziuk, \cite{Dz88} was to use a triangulation of the surface. We refer to \cite{DzEl13} and the references therein for an overview.

%
 
\paragraph{Organization of the Paper.}  In Section 2 we introduce 
the model problem, in Section 3 we define the cut finite element method, in Section 4 we derive our main theoretical results, in Section 5 we prove 
a bound on the condition number, and in Section 6 we present numerical results that confirm our theoretical results.

\section{The Model Problem}
\subsection{The Surface} Let $\Gamma$ be a smooth 
surface without boundary embedded in $\IR^3$ with signed 
distance function $\rho$, 
such that the exterior unit normal to the surface is given by 
$n = \nabla \rho$. We let $p:\IR^3 \rightarrow \Gamma$ be 
the closest point mapping. Then there is a $\delta_0> 0 $ 
such that $p$ maps each point in 
$U_{\delta_0}(\Gamma)$ to precisely one point on $\Gamma$. 
Here $U_\delta(\Gamma) = \{ x \in \IR^3 : |\rho(x)| < \delta\}$ is the 
open tubular neighborhood of $\Gamma$ of thickness $\delta$.

\subsection{Tangential Calculus} For each function $u$ on 
$\Gamma$ we let the extension $u^e$ to the neighborhood 
$U_{\delta_0}(\Gamma)$ be defined by the pull back $u^e = u \circ p$. 
For a function $u: \Gamma \rightarrow \IR$ we then define the 
tangential gradient
\begin{equation}
\nablas u = \Ps \nabla u^e
\end{equation}
where $\Ps = I - n \otimes n$, with $n=n(x)$, is the projection 
onto the tangent plane $T_{x}(\Gamma)$. We also define the 
surface divergence 
\begin{equation}
\divs(u) = \text{tr}(u \otimes \nablas ) 
= \text{div}(u^e) - n \cdot (u^e \otimes \nabla)\cdot n 
\end{equation}
where $(u\otimes \nabla)_{ij} = \partial_j u_i$. It can be shown 
that the tangential derivative does not depend on the particular 
choice of extension.

\subsection{The Convection-Diffusion Problem on $\boldsymbol{\Gamma}$}

The strong form of the convection problem on $\Gamma$ takes 
the form: find $u:\Gamma \rightarrow \IR$ such that
\begin{alignat}{2}\label{eq:conva}
L u = \beta \cdot \nablas u + \alpha u - \epsilon \Delta_\Gamma u &= f \qquad &\text{on $\Gamma$}
\end{alignat}  
where $\beta:\Gamma \rightarrow \IR^3$ is a given tangential vector 
field, $\alpha:\Gamma \rightarrow \IR$, $\epsilon\in \IR_+$, and 
$f:\Gamma \rightarrow \IR$ are given functions. 

We assume that the coefficients 
$\alpha$ and $\beta$ are smooth and that 
there is a constant $\alpha_0$ such that 
\begin{equation}\label{eq:assumdivbeta}
0< \alpha_0 \leq \inf_{x\in \Gamma} (\alpha(x) - \frac{1}{2} \divs \beta(x) ) 
\end{equation} 
We note that using Green's formula and assumption 
(\ref{eq:assumdivbeta}) we obtain the estimate
\begin{equation}
(Lv,v)_\Gamma = ((\alpha - \frac{1}{2}\divs \beta)v,v)_\Gamma 
+ \epsilon \|\nablas v\|^2_\Gamma
\geq \alpha_0 \|v\|_\Gamma^2 + \epsilon \| \nablas v \|^2_\Gamma 
\end{equation}

We introduce the space 
\begin{equation}
V= \left \{ \begin{array}{ll}
H^1(\Gamma)  &\textrm{if $\alpha>0$} \\
\{v \in H^1(\Gamma) : \int_{\Gamma} v=0 \}  &\textrm{if $\alpha = 0$}
\end{array}
\right .
\end{equation}
The weak formulation of~\eqref{eq:conva} takes the form: find $u\in V$ such that 
\begin{equation}\label{eq:weakform}
a(u,v)=(f,v)_{\Gamma}
\end{equation}
where 
\begin{equation}\label{eq:a-small}
a(v,w) = (\beta \cdot\nablas v,w)_{\Gamma} 
+ (\alpha v,w)_{\Gamma} + \epsilon (\nablas v, \nablas w)_{\Gamma}
\end{equation}
In the case $\epsilon>0$ we may conclude using Lax--Millgram's lemma
that there is a unique solution to (\ref{eq:weakform}). In the case
$\epsilon=0$ we refer to \cite{BuHaLaZa15} for an existence result.
\begin{rem}
We will be interested in the behavior of the numerical method for small
$\epsilon$, but will assume that the coefficients  $\alpha_0$, $\alpha$ and $\beta$
are all $O(1)$.
\end{rem}

 \section{The Finite Element Method}

\subsection{The Discrete Surface}
Let $\Omega_0\subset \IR^3$ be a polygonal domain that contains 
$U_{\delta_0}(\Gamma)$ and let $\{\mcT_{0,h}, h \in (0,h_0]\}$ 
be a family of quasiuniform partitions of $\Omega_0$ into shape 
regular tetrahedra with mesh parameter $h$. Let $\Gamma_h\subset \Omega_0$ be a connected surface such that $\Gamma_h \cap T$ is 
a subset of some hyperplane for each $T\in \mcT_{0,h}$ and 
let $n_h$ be the piecewise constant unit normal to $\Gamma_h$. 

\paragraph{Assumption A.} The family $\{\Gamma_h: h \in (0,h_0]\}$ 
approximates $\Gamma$ in the following sense: 
\begin{itemize}
\item $\Gamma_h \subset U_{\delta_0}(\Gamma)$, $\forall h \in (0,h_0]$, and the closest point mapping $p:\Gamma_h \rightarrow \Gamma$ is a 
bijection.
\item The following estimates hold
\begin{equation}\label{assum:geom}
\|\rho\|_{L^\infty(\Gamma_h)} \lesssim h^2, \qquad 
\|n - n_h \|_{L^\infty(\Gamma_h)} \lesssim h
\end{equation}
\end{itemize}
We introduce the following notation for the geometric entities involved 
in the mesh
\begin{align}
\mcT_h &= \{T\in \mcT_{h,0} : \overline{T}\cap \Gammah \neq \emptyset\}
\\
\mcF_h &= \{F = (\overline{T}_1 \cap \overline{T}_2)
\setminus 
\partial (\overline{T}_1 \cap \overline{T}_2) : T_1,T_2 \in \mcT_h\}
\\
\mcK_h &=\{K = T \cap \Gammah : T \in \mcT_h\}
\cup \{F \in \mcF_h: F \subset \Gamma_h\}
\\
\mcE_h &= \{E= \partial K_1 \cap \partial K_2 : K_1,K_2 \in \mcK_h\} 
\end{align}
We also use the notation $\omega^l = p(\omega) 
= \{p(x)\in \Gamma : x \in 
\omega \subset \Gamma_h\}$, in particular $\mcK_h^l = 
\{K^l : K\in \mcK_h^l\}$ is a partition of $\Gamma$.
We will use the notation $\|v \|_{\omega}$ for the $L^2$ norm over the set $\omega$.
\subsection{The Discrete Formulation}
Let $V_h$ be the space of continuous piecewise linear 
functions defined on $\mcT_h$.
To write the variational form of~\eqref{eq:conva} in the discrete spaces we
introduce finite dimensional approximations of the physical parameters
$\alpha$ and $\beta$, $\alpha_h$ and $\beta_h$. The assumptions made
on these approximations will be detailed later.
 
 The finite element method then
takes the form: find $u_h \in V_h$ such that 
\begin{equation}\label{eq:fem}
A_h(u_h,v) = L_h(v) \quad \forall v \in V_h
\end{equation}
The forms are defined by
\begin{equation}\label{eq:Ah}
A_h(v,w) = a_h(v,w) + s_{h,1}(v,w) + s_{h,2}(v,w)
\end{equation}
with 
\begin{align}\label{eq:ah-small}
a_h(v,w) &= (\beta_h \cdot\nablash v,w)_{\mcK_h} 
+ (\alpha_h v,w)_{\mcK_h} + \epsilon (\nablash v, \nablash w)_{\mcK_h}
\\ \label{eq:sh-one}
s_{h,1}(v,w) &= \tau_1 h (\beta_h \cdot\nablash v + \alpha_h v, 
\beta_h \cdot \nablash w)_{\mcK_h}
\\ \label{eq:sh-two}
s_{h,2} (v,w) &= \tau_2 h^\gamma (n_h\cdot \nabla v, n_h \cdot \nabla w )_{\mcTh}
\end{align}
where $\nablash v = \Psh \nabla v = (I - n_h\otimes n_h) \nabla v$ is the tangential gradient on $\Gammah$,  $\alpha_h$ and $\beta_h$ are discrete approximations of $\alpha$ and $\beta$, which satisfy Assumptions B and C below, and
\begin{equation}\label{eq:Lh}
L_h(v) = l_h(v) + l_{\shone}(v)=(f^e,v)_\mcKh + \tau_1 h (f^e,\beta_h \cdot \nablash v )_\mcKh
\end{equation}
The two terms $s_{h,1}$ and $s_{h,2}$ are least squares terms that we
add to stabilize the method. The streamline diffusion stabilization
$s_{h,1}$ includes the weighting parameter,
\begin{equation}\label{eq:tau1}
\tau_1=\left \{ 
\begin{tabular}{ll}
$c_\tau \beta_\infty^{-1}$ &  if $\beta_\infty h  \ge \epsilon$ (high Peclet number regime)  \\
$c_\tau h \epsilon^{-1}$ &  if $\beta_\infty h  \leq \epsilon$ (low Peclet number regime) 
\end{tabular}\right.
\end{equation}
where $c_\tau$ is positive constant, observe that this can be more
compactly written $\tau_1 = c_\tau \min(\beta_{\infty}^{-1},  h
\epsilon^{-1})$. The normal gradient stabilization term $s_{h,2}$
has two parameters: $\tau_2$ that may be chosen as $\tau_1^{-1}$ and the
power of $h$, where  $\gamma=1$, is a suitable choice for which we can
show optimal order a priori error estimates when the solution is smooth, also for vanishing diffusion, and optimal order condition number.  

\begin{rem}
From the stability estimate and consistency we have that $0\leq \gamma < 2$. 
More precisely,  the derivation of the coercivity of $A_h$ provides an upper bound on $\gamma$, guaranteeing that the stabilization is strong enough, and the consistency result provides the lower bound, guaranteeing that the stabilization is weak enough not to affect the optimal order of convergence (see the a priori estimate). The condition number estimate, see Theorem \ref{thm:condition-number}, however shows that the best choice is $\gamma=1$. This is the largest $\gamma$, and thus the weakest stabilization, that gives optimal order condition number.
\end{rem}

\begin{rem}\label{rem:fullgrad} In \cite{BuHaLaMaZa16} the so called 
full gradient stabilization was proposed and analyzed for the Laplace-Beltrami operator. Applying the same idea for the convection-diffusion problem we find that a suitable full gradient stabilization term takes the form 
\begin{equation}\label{eq:fullgrad-stab}
h(h+\epsilon) (\nabla v,\nabla w)_{\mcTh}
\end{equation}
where the powers of $h$ is given by the a priori error estimate. We note 
that the normal control provided by the full gradient stabilization term is weaker compared to the normal gradient due to the different $h$ scalings. To prove coercivity, at least using straight forward estimates, 
we will need to use the antisymmetric formulation of the convection term. See Remark \ref{rem:antisym} below. Note also that the $h$ scaling 
here is fixed while in the normal gradient stabilization we have 
some flexibility.
\end{rem}


\section{Preliminary Results}

\subsection{Extension and Lifting of Functions}

In this section we summarize basic results concerning 
extension and liftings of functions. We refer to 
\cite{BuHaLa15}, \cite{HaLaLa16}, and \cite{De09} for 
further details.
\paragraph{Extension.}
Recalling the definition $v^e = v \circ p$ of the extension 
and using the chain rule we obtain the identity 
\begin{equation}\label{eq:tanderext}
\nablash v^e = B^T \nablas v  
\end{equation}
where 
\begin{equation}\label{Bmap}
B = \Ps(I - \rho \kappa)\Psh : T_x(K)\rightarrow
T_{p(x)} (\Gamma)
\end{equation}
and $\kappa = \nabla \otimes \nabla \rho$ is the curvature tensor 
(or second fundamental form) which may be expressed in the form
\begin{equation}\label{Hform}
\kappa(x) = \sum_{i=1}^2 \frac{\kappa_i^e}{1 + \rho(x)\kappa_i^e} a_i^e \otimes a_i^e
\end{equation}
where $\kappa_i$ are the principal curvatures with corresponding
orthonormal principal curvature vectors $a_i$, see \cite[Lemma 14.7]{GiTr83}. We note that there is $\delta>0$ such that the uniform 
bound 
\begin{equation}
\|\kappa \|_{L^\infty(U_\delta(\Gamma))}\lesssim 1
\end{equation}
holds. Furthermore,  
$B:T_{x}(K) \rightarrow T_{p(x)} (\Gamma)$ is  invertible for 
$h \in (0,h_0]$ with $h_0$ small enough, i.e, there is
 $B^{-1}:   T_{p(x)} (\Gamma) \rightarrow  T_{x}(K)$ such that 
 \begin{equation}
 B B^{-1} = P\textcolor{red}{_\Gamma}, \qquad B^{-1} B = P\textcolor{red}{_{\Gamma_h}}
 \end{equation}
See \cite{HaLaLa16} for further details.

\paragraph{Lifting.}
The lifting $w^l$ of a function $w$ defined on $\Gamma_h$ 
to $\Gamma$ is defined as the push forward
\begin{equation}
(w^l)^e = w^l \circ p = w \quad \text{on $\Gamma_h$}
\end{equation}
For the derivative it follows that 
\begin{equation}
\nablash w = \nablash (w^l)^e = B^T \nablas (w^l) 
\end{equation}
and thus 
\begin{equation}\label{eq:tanderlift}
 \nablas (w^l) = B^{-T} \nablash w 
\end{equation}

\paragraph{Estimates Related to $\boldsymbol{B}$.}
Using the uniform bound  $\|\kappa\|_{U_{\delta_0}(\Gamma)} \lesssim 1$ and 
the bound $\|\rho\|_{L^\infty(\Gammah)} \lesssim h^{2}$ from the geometry 
approximation assumption it follows that
\begin{alignat}{2}
\label{B-uniform-bounds}
  \| B \|_{L^\infty(\Gamma_h)} &\lesssim 1,
 \qquad &\| B^{-1} \|_{L^\infty(\Gamma)} 
 &\lesssim 1
\\
\label{eq:B-PPh}
\| \Ps\Psh - B \|_{L^\infty(\Gamma)} 
 &\lesssim h^{2}, 
 \qquad\quad
 &\| \Psh\Ps -  B^{-1} \|_{L^\infty(\Gamma_h)} 
 &\lesssim h^{2}
\end{alignat}
For the surface measures on $\Gamma$ and $\Gammah$ 
we have the identity  
\begin{equation}\label{eq:measure}
d \Gamma = |B| d \Gammah
\end{equation}
where $|B| =| \text{det}(B)|$ is the absolute value 
of the determinant of $B$ and we have the following 
estimates
\begin{equation}\label{eq:B-detbound}
\left\| 1 - |B| \right\|_{L^\infty(\Gamma_h)} \lesssim h^{2}, 
\quad \left\| 1 - |B^{-1}| \right\|_{L^\infty(\Gamma_h)} \lesssim h^{2}, \quad\left\| |B| \right\|_{L^\infty(\Gamma_h)} \lesssim 1,
\quad \left\| |B|^{-1} \right\|_{L^\infty(\Gamma_h)} \lesssim 1
\end{equation}

\paragraph{Norm Equivalences.} We have
\begin{align}
\label{eq:normequ}
\| v \|_{L^2(\Gamma)}
\sim \| v \|_{L^2(\Gammah)} \qquad \text{and} \qquad
\| \nablas v \|_{L^2(\Gamma)} &\sim
\| \nablash v \|_{L^2(\Gammah)}
\end{align}

\subsection{Assumptions}
We make the following assumptions about the approximations $\alpha_h$ and $\beta_h$:
\paragraph{Assumption B.} 
\begin{itemize} 
\item There is a constant $\alpha_{h_0,0}$ such that 
for all $h\in (0,h_0]$,
\begin{equation}\label{eq:assum-ah}
0< \alpha_{h_0,0} \leq \alpha_h - \frac{1}{2}\divsh \beta_h 
\end{equation} 
\item There is a constant such that for all $h\in (0,h_0]$,
\begin{equation}\label{eq:assum-jump-betah}
\| [\nu_h \cdot \beta_h ] \|_{L^\infty(\mcEh)} \lesssim h^{2}
\end{equation}
\end{itemize}
Here
\begin{equation}
[\nu_h \cdot \beta_h ] 
= \nu_{h,K_1} \cdot  \beta_h + \nu_{h,K_2} \cdot  \beta_h
\end{equation}
where $\nu_{h,K_i}$ denotes the unit vector orthogonal to the edge $E \in \mcEh$ shared by the elements $K_1$ and $K_2$, tangent and exterior to $K_i$, $i=1,2$.

\paragraph{Assumption C.} \textcolor{black}{$\beta_h$ and $\alpha_h$ are
$L^\infty$-stable discrete approximations of $\beta^e$ and $\alpha^e$
in $V^h$, sucht that $\beta_h = P_{\Gamma_h} \beta_h$.}  There are constants such that for all 
$h\in (0,h_0]$,
\begin{align}\label{eq:assum-coefficients-accuracy}
\| \beta - |B|^{-1} B \beta_h\textcolor{red}{^l} \|_{L^\infty(\Gamma)} \lesssim h^2, 
\qquad 
\|\alpha - |B|^{-1}  \alpha_h^l \|_{L^\infty(\Gamma)} \lesssim h^2
\end{align}
\begin{rem} Using the bounds (\ref{eq:B-detbound}), (\ref{eq:B-PPh}) and the identity 
(\ref{Bmap}) we note that
\begin{equation}
\| \beta - |B|^{-1} B \beta_h^l \|_{L^\infty(\Gamma)}  
\lesssim
\| \beta -  P_\Gamma \beta_h^l \|_{L^\infty(\Gamma)} + O(h^2) 
\end{equation}
and using (\ref{eq:B-detbound}) we have
\begin{equation}
\|\alpha - |B|^{-1}  \alpha_h^l \|_{L^\infty(\Gamma)} \lesssim
\|\alpha -  \alpha_h^l \|_{L^\infty(\Gamma)} + O(h^2)
\end{equation}
Thus we conclude that (\ref{eq:assum-coefficients-accuracy}) is 
equivalent to the simplified assumptions
\begin{equation}\label{eq:assum-coefficients-accuracy-simplified}
\| \beta^e -  P_\Gamma \beta_h \|_{L^\infty(\Gamma)}  
\lesssim h^2, 
\qquad
\|\alpha -  \alpha_h^l \|_{L^\infty(\Gamma)} 
\lesssim h^2
\end{equation}
\end{rem}
For a detailed discussion on how to construct $\alpha_h$ and $\beta_h$
with the above properties under the regularity assumption $\alpha \in
C^2(\Gamma)$ and $\beta \in [C^2(\Gamma)]^2$ see \cite{BuHaLaZa15}.
\subsection{Inequalities for stability and approximation} 
Here we will recall two inequalities that are useful in the analysis
of the stabilized method. The first that was originally introduced in
\cite{BuHaLaMa16} is a Poincar\'e type inequality showing that the
$L^2$-norm of the finite element solution in the bulk can be
controlled by the $L^2$-norm over the discrete surface plus the
normal component of the bulk gradient, scaled with $h$. The second is
a trace inequality showing that the scaled $L^2$-norm of the finite element
solution over the edges of the tesselation of the discrete surface can
be bounded by the $L^2$-norm over the surface, plus the scaled normal
stabilization term.
\begin{lem}\label{lem:poincare-normalstab}
There is a constant such that for all $v\in V_h$,
\begin{equation}\label{eq:poincare-normalstab}
\| v \|^2_{\mcTh} 
\lesssim 
h \| v \|^2_{\mcKh} 
+ 
h^2 \| n_h \cdot \nabla v \|^2_{\mcTh} 
\lesssim 
h \| v \|^2_{\mcKh} +  \gamma_2h^{2} \| v \|^2_{s_{h,2}}
\end{equation}
\end{lem}
\begin{proof}
See \cite{BuHaLaMa16} Proposition 8.8.
\end{proof}

\begin{lem} \label{lem:inverseedge}
There is a constant such that for all $v\in V_h$,
\begin{equation}\label{eq:inverseedge}
h \| v \|^2_\mcEh \lesssim \| v \|^2_\mcKh +   \tau_2h\| v \|^2_{s_{h,2}}
\end{equation} 
There is a constant such that for all $v \in V_h + H^2(\mathcal{T}_h)$,
\begin{equation}\label{eq:inverseedge2}
h \| v \|_\mcEh \lesssim \| v \|_\mcTh +   h\| \nabla v \|_\mcTh + h^2
|v|_{H^2(\mcTh)}
\end{equation}
where $|\cdot|_{H^2(\mcTh)}$ denotes the broken (or elementwise) $H^2$-seminorm over $\mcTh$.
\end{lem} 
\begin{proof}
Using an inverse estimate followed by (\ref{eq:poincare-normalstab}) 
we obtain
\begin{align}
h \| v \|^2_\mcEh &\lesssim h^{-1} \|v \|^2_\mcTh 
\lesssim \| v \|^2_\mcKh + \tau_2 h \| v \|^2_{s_{h,2}}
\end{align}
The inequality \eqref{eq:inverseedge2} follows by applying the standard trace inequality
\[
\|v\|_F \leq C_T( h^{-\frac12} \|v\|_K + h^{\frac12} \|\nabla v\|_K)
\]
twice, for each edge in $\mcEh$, once to pass to a face in the bulk
mesh and a second time to pass to an element in the bulk mesh.
\end{proof}

\subsubsection{Interpolation error estimates} 
For the error analysis in the next section we will use the following
mesh-dependent norm:
\begin{equation}\label{eq:meshnorm}
\tn v \tn^2_h 
= 
\| v \|^2_\mcKh + \epsilon \|\nablash v \|^2_\mcKh 
+ \tau_1  h \|\beta_h \cdot \nablash v \|^2_\mcKh
+ \tau_2 h^{\gamma} \| v \|^2_{\shtwo}
\end{equation}
where
\begin{equation}
 \| v \|^2_\shtwo = (n_h\cdot \nabla v, n_h \cdot \nabla w )_{\mcTh}
\end{equation} 
Let 
$\pi_h: L^2(\mcTh)\rightarrow V_h$ be the Scott-Zhang 
interpolant. Using the stability 
\begin{equation}
\| v^e \|_{H^s(U_{\delta_0}(\Gamma))}
\lesssim \delta^{1/2} \| v \|_{H^s(\Gamma)}
\end{equation}
of the extension we obtain the interpolation error estimate 
\begin{equation}\label{eq:interpol-basic}
\| u - (\pi_h u^e)^l \|_{H^m(\Gamma)} \sim \| u^e - \pi_h u^e \|_{H^m(\Gammah)} 
\lesssim 
h^{s-m} \|u \|_{H^s(\Gamma)}\quad m \in \{0,1\}, m\leq s \leq 2
\end{equation}
Furthermore, we have the energy norm estimate 
\begin{equation}\label{eq:interpol-energy}
 \tn u^e - \pi_h u^e \tn_h 
\lesssim \max(\beta_\infty^{\frac12} h^{\frac32}, \epsilon^{\frac12} h) \|u \|_{H^2(\Gamma)} 
\end{equation}
Considering the definition of \eqref{eq:meshnorm} we see that the
bound for the first two terms is an immediate consequence of
\eqref{eq:interpol-basic}. For the two last terms of
\eqref{eq:meshnorm}, that are related to the stabilization, we see that
using \eqref{eq:interpol-basic} and the definition of $\tau_1$
we have
\begin{multline*}
\tau_1^{\frac12}  h^{\frac12}  \|\beta_h \cdot \nablash (u^e - \pi_h u^e ) \|_\mcKh \leq
C \tau_1^{\frac12} \beta_\infty h^{\frac32}  \|u \|_{H^2(\Gamma)} \leq C
\min(\beta_\infty^{-\frac12}, h^{\frac12} \epsilon^{-\frac12})
\beta_\infty  h^{\frac32}  \|u \|_{H^2(\Gamma)} \\
\leq C \beta_\infty^{\frac12} h^{\frac32}  \|u \|_{H^2(\Gamma)} 
\end{multline*}
where we used that $\min(\beta_\infty^{-\frac12}, h^{\frac12}
\epsilon^{-\frac12}) \beta_\infty^{\frac12} \leq 1$ in the last
inequality. Using  the definition of the normal stabilization we also
get
\[
\tau^{\frac12}_2 h^{\frac{\gamma}{2}} \| u^e - \pi_h u^e\|_{\shtwo}
\leq \tau^{\frac12}_2 h^{\frac{\gamma}{2}+1}  \|u \|_{H^2(\Gamma)}.
\]
We obtain the desired result for $\gamma \ge 1$ and 
$$
\tau_2 \sim \max(\beta^{\frac12}_\infty, \epsilon h^{-1}) \sim \tau_1^{-1}
$$
\section{Apriori Error Estimates}
In this section we will prove the main result of this paper: an
optimal error estimate in the streamline derivative norm and an
estimate that is suboptimal with $O(h^{\frac12})$ for the error in the
$L^2$-norm. To give some structure to this result we first prove
coercivity, which also establishes the existence of the discrete
solution, then continuity and finally estimates of the geometrical
error and consistency.

\subsection{Coercivity} 
Compared to a standard coercivity result for a problem set in the flat
domain we must here control the terms appearing due to jumps in the discrete
approximation of $\beta$ over element faces. To obtain this control we
need to use equation
\eqref{eq:assum-jump-betah} of Assumption B and the normal grandient stabilization. 
\begin{lem} \label{lem:coercivity} For all $v\in V_h$ 
and $h \in(0,h_0]$ we have
\begin{equation}\label{eq:coercivity}
\min\Big(1, \alpha_{h_0,0}-Ch_0, 1-Dh_0^{2-\gamma}\Big) \tn v \tn_h^2 \lesssim A_h(v,v) \qquad \forall v \in V_h 
\end{equation}
where $C$ and $D$ are positive constants. For $h_0$ small enough and $\gamma <2$ we have that the constant $c_{coer}=\min\Big(1, \alpha_{h_0,0}-Ch_0, 1-Dh_0^{2-\gamma}\Big) >0$.
\end{lem}
\begin{proof} We have 
\begin{equation}
A_h(v,v) = a_h(v,v) + s_{h,1}(v,v) + s_{h,2}(v,v) = I + II + III
\end{equation}
\paragraph{Term $\bfI$.} Using Assumptions (\ref{eq:assum-ah}) and 
(\ref{eq:assum-jump-betah}) we obtain
\begin{align}
a_h(v,v) &= (\beta_h \cdot \nablash v,v)_\mcKh 
+ (\alpha_h v,v )_\mcKh 
+ \epsilon(\nabla v,\nabla w)_\mcKh  
\\
&=\underbrace{((\alpha_h - \frac{1}{2}\divsh \beta_h) v,v)_\mcKh}_{
\geq \alpha_{h_0,0} \| v \|_\mcKh^2 } 
+ 
\underbrace{([\nu_h \cdot \beta_h] v,v )_\mcEh}_{\geq - \Big( C_0 h \|
  v \|^2_\mcKh + C_1 \tau_2h^{2} \| v \|^2_{s_{h,2}}\Big)}
+ \epsilon \|\nablash v \|^2_\mcKh  
\end{align}
where we used (\ref{eq:inverseedge}) to conclude that 
\begin{align}
([\nu_h \cdot \beta_h] v,v )_\mcEh 
&\lesssim 
h^2 \| v \|^2_\mcEh 
\leq
C_0 h \| v \|^2_\mcKh +C_1 \tau_2 h^{2} \|v\|^2_{s_{h,2}}  
\end{align}
Observe that the constant $C_0 \sim |\beta|_{W^2,\infty(\Gamma)}$ and
$C_1 \sim  |\beta|_{W^2,\infty(\Gamma)}/\tau_2$
\paragraph{Term $\bfI \bfI$.} Expanding and estimating the 
second term using the Cauchy-Schwarz inequality and the bound 
$2ab \leq a^2 + b^2$ we obtain
\begin{align}
s_{h,1}(v,v) &= \tau_1 h (\nablashb v + \alpha_h v, \nablashb v )_\mcKh  
\\
&\geq \tau_1 h \| \nablashb v \|^2_\mcKh  
- \tau_1  h \| \alpha_h v \|_\mcKh
\| \nablashb v \|_\mcKh  
\\
&\geq \frac{\tau_1 }{2} h \| \nablashb v \|^2_\mcKh  
- \frac{\tau_1 }{2}  h \|\alpha_h \|^2_{L^\infty(\mcKh)} \| v \|^2_\mcKh
\end{align}

\paragraph{Term $\bfI \bfI \bfI$.} We directly have 
\begin{equation}
s_{h,2}(v,v) =\tau_2h^\gamma \| v \|^2_\shtwo
\end{equation}

\paragraph{Conclusion.}
Collecting the estimates we obtain
\begin{align}
A_h(v,v)&\geq \Big(\alpha_{h_0,0} - C_0 h - h \frac{\tau_1}{2}\|\alpha_h \|^2_{L^\infty(\mcKh)}\Big) \| v \|^2_\mcKh 
\\ \nonumber
& \qquad + \epsilon \| \nablash v \|^2_\mcKh 
\\ \nonumber
& \qquad + \frac{\tau_1}{2} h \| \nablashb v \|^2_\mcKh  
\\
&\qquad + \Big(1 - C_1h^{2-\gamma} \Big)
\tau_2h^\gamma\| v \|^2_\shtwo  
\\
&\gtrsim 
\min\Big(1, \alpha_{h_0,0}-Ch_0, 1-Dh_0^{2-\gamma}\Big)
\tn v \tn^2_h 
\end{align}
for $h \in (0,h_0]$ with $h_0$ small enough.
\end{proof}
\begin{rem}
Note that from the coercivity proof we see that the condition
$\gamma<2$ is necessary in order to be able to control instabilities
due to the approximated transport velocity field. We also remark that
the smallness assumption on $h$  can be related to the physical parameters hidden in the constants
$C_0$ and $C_1$. If $\gamma=1$, the conditions that must be satisfied
is that $ |\beta|_{W^2,\infty(\Gamma)}
\beta_{\infty}^{-1} h$ is small for convection domainted flows and $ |\beta|_{W^2,\infty(\Gamma)}
h^2 \epsilon^{-1}$ is small for diffusion dominated flow.
 \end{rem}
\begin{rem}\label{rem:antisym} Note that instead starting from the antisymmetric 
discretization of the convection term 
\begin{equation}
\frac{1}{2}\Big( (\nablashb v, w)_\mcKh - (v,\nablashb w)_\mcKh\Big)
- \frac{1}{2}((\divsh \beta_h) v,w)_\mcKh
\end{equation}
we do not have to use partial integration in Term $I$, which 
simplifies the argument since we immediately obtain
\begin{align}
A_h(v,v) &= ((\alpha_h - \frac{1}{2}\divsh \beta_h) v,v)_\mcKh 
+\epsilon (\nablash v, \nablash v )_\mcKh 
+s_{h,1}(v,v) + s_{h,s}(v,v) 
\end{align}
We note that \cite{OlReXu14b} uses the skew symmetric form 
and may thus establish coercivity without using the stabilization term $s_{h,2}$.
With stabilization we find that we may use the  standard or the antisymmetric formulation of the convection term. However, partial integration must still be used to prove optimal a priori error estimate, see the proof of the continuity result in the next section. 
\end{rem} 

\subsection{Continuity}
We now prove a continuity result. 
\begin{lem} \label{lem:continuity}There is a constant such that for
  all $\eta \in V_h + H^2(\mcTh)$, $v\in V_h$, 
we have
\begin{equation}\label{eq:continuity}
A_h (\eta,v) \lesssim  \tau_1^{-1/2} h^{-1/2} \| \eta \|_\mcKh \tn v
\tn_h \\
+(\tn \eta \tn_h +C_\beta \max(1, h^\frac{2-\gamma}{2})\|\eta\|_*) \tn v \tn_h
\end{equation}
where $\|\eta\|_* =  h^{\frac12} \Big(\|\eta \|_\mcTh + h\| \nabla \eta \|_\mcTh + h^2|\eta|_{H^2(\mcTh)} \Big)$
\end{lem}
\begin{proof} 
We have
\begin{align}
A_h (\eta,v)&=a_h(\eta,v) + s_{h,1}(\eta , v) + s_{h,2}(\eta,v) 
= I + II + III 
\end{align}
\paragraph{Term $\bfI$.} 
Using partial integration on the discrete surface 
followed by the Cauchy-Schwarz inequality we obtain
\begin{align}
I &= ( \alpha_h \eta , v )_\mcKh 
+ (\nablashb \eta, v)_\mcKh 
+ \epsilon (\nablash \eta, \nablash v)_\mcKh  
\\
&=( (\alpha_h -\divsh \beta_h )\eta , v )_\mcKh 
- (\eta, \nablashb v)_\mcKh 
\\ \nonumber
&\qquad + ([\nu_h \cdot \beta_h]\eta,v)_\mcEh 
+ \epsilon (\nablash \eta, \nablash v)_\mcKh
\\
&\leq (\|\alpha_h \|_{L^\infty(\mcKh)}  
+ \|\divsh \beta_h \|_{L^\infty(\mcKh)} )\|\eta\|_\mcKh \| v \|_\mcKh 
\\ \nonumber
&\qquad + \tau_1^{-\frac12} h^{-1/2}\|\eta\|_\mcKh  \underbrace{\tau_1^{\frac12} h^{1/2} \|\nablashb v\|_\mcKh}_{
\lesssim  \tn v \tn_h}
\\ \nonumber
&\qquad + \underbrace{\|[\nu_h \cdot \beta_h]\|_{L^\infty(\mcEh)} 
\|\eta\|_\mcEh \|v\|_\mcEh}_{\bigstar} 
\\ \nonumber
&\qquad 
+ \epsilon \|\nablash \eta\|_\mcKh \|\nablash v\|_\mcKh
\\
&\lesssim  \tau_1^{-1/2}h^{-1/2} \| \eta \|_\mcKh \tn v \tn_h+
  (\tn \eta \tn_h+ C_\beta \max(1, h^\frac{2-\gamma}{2})  \|\eta\|_*)\tn v \tn_h
\end{align} 
For $\bigstar$ we used \eqref{eq:inverseedge2} and
\eqref{eq:inverseedge} to obtain the bound
\begin{align}
\bigstar 
&= \underbrace{\|[\nu_h \cdot \beta_h]\|_{L^\infty(\mcEh)}}_{\leq
  C_\beta h^2} 
\|\eta\|_\mcEh \|v\|_\mcEh
\\
&\leq C_\beta  h^{1/2} h \|\eta\|_\mcEh h^{1/2} \|v\|_\mcEh
\\
&\lesssim C_\beta  h^{\frac12} \Big(\|\eta \|_\mcTh + h\| \nabla \eta \|_\mcTh + h^2|\eta|_{H^2(\mcTh)} \Big)
\underbrace{\Big(\|v \|^2_\mcKh + \tau_2h\| v \|^2_{s_{h,2}} \Big)^{1/2}}_{\lesssim \max(1, h^\frac{1-\gamma}{2} ) \tn v \tn_h } 
\\
&\lesssim \max(1, h^\frac{1-\gamma}{2}) \| \eta\|_{*} 
\tn v \tn_h
\end{align}
\paragraph{Term $\bfI \bfI$.} 
\begin{align}
II &=\tau_1 h (\alpha_h \eta + \beta_h \cdot \nablash \eta, \beta_h \cdot \nablash v)_\mcKh
\\
&\lesssim \tau_1^{1/2} h^{1/2} \|\alpha_h \eta + \beta_h \cdot \nablash \eta\|_\mcKh \tau_1^{1/2} h^{1/2} \|\beta_h \cdot \nablash v \|_\mcKh
\\
&\lesssim  \tn \eta \tn_h \tn v \tn_h
\end{align}

\paragraph{Term $\bfI \bfI \bfI$.}
\begin{align}
III &= \tau_2^{1/2} h^{\gamma/2} \| n_h \cdot \nabla \eta \|_\mcTh  
 \tau_2^{1/2} h^{\gamma/2} \|n_h\cdot \nabla v\|_\mcTh
\\
&\lesssim  \tn \eta \tn_h  \tn v \tn_h
\end{align}
Collecting the estimates of terms $I$-$III$ proves the lemma. 
\end{proof}

\subsection{Geometric Error Estimates} 
We have the following estimates: for $v \in H^1(\Gamma)$ and 
$w \in V_h$, 
\begin{align}\label{eq:quad-ah}
|a(v,w^l) - a_h(v^e,w)|\lesssim 
h^2 \| v \|_{H^1(\Gamma)} \| w \|_{\mcKh} 
\end{align}
and for $w \in V_h$,
\begin{align}\label{eq:quad-lh}
|l(w^l) - l_h (w) | \lesssim h^2 \|f \|_\Gamma \| w \|_{\mcKh}
\end{align}
\paragraph{Verification of (\ref{eq:quad-ah}).} Using 
(\ref{eq:tanderext}) and changing domain of integration 
from $\Gammah$ to $\Gamma$ we obtain
\begin{align}
(\beta_h \cdot \nablash v^e, w)_\mcKh
- (\beta \cdot \nablas v, w^l)_\Gamma
&= 
(|B|^{-1}(B \beta_h^l \cdot \nablas v), w^l)_\Gamma 
- (\beta \cdot \nablas v, w^l)_\Gamma
\\
&= 
((|B|^{-1}B \beta_h^l - \beta) \cdot \nablas v), w^l)_\Gamma 
\\
&\lesssim 
h^2 \|\nablas v \|_\Gamma \| w \|_\mcKh
\end{align}
where we used Assumption C and (\ref{eq:normequ}) in the last 
step. Using the same approach we obtain 
\begin{align}
(\alpha_h v^e, w)_\Gammah - (\alpha v,w^l)_\Gamma 
&=
((\alpha - |B|^{-1} \alpha_h^l) v,w^l)_\Gamma
\\
&\lesssim 
h^2 \|v\|_{\Gamma} \| w\|_\mcKh
\end{align}
Finally (\ref{eq:quad-lh}) follows in the same way.

\subsection{Consistency}
We now estimate the consistency error which depends on the geometric error. 

\begin{lem} \label{lem:consistency}There is a constant such that for all $v\in V_h$ and $u\in H^2(\Gamma)$
\begin{equation}\label{eq:consistency}
A_h ( u^e,v)-L_h(v) \lesssim \Big( h^2 + \tau_2^{1/2} h^{(\gamma+3)/2 } + \tau_1^{1/2}(h^{5/2} + \epsilon h^{1/2} )\Big)  \| u  \|_{H^2(\Gamma)} \tn v \tn_h 
+h^2 \|f\|_{\Gamma} \tn v \tn_h 
\end{equation}
\end{lem}
\begin{proof}
We have the identity
\begin{align}
A_h ( u^e,v)-L_h(v)=\underbrace{a_h (u^e,v)  - l_h(v)}_{I}
+ \underbrace{s_{h,1}(u^e,v) - l_{s_{h,1}}(v)}_{II} + 
\underbrace{s_{h,2}(u^e,v)}_{III} 
 \end{align}

\paragraph{Term $\bfI$.} Using the geometry error estimates (\ref{eq:quad-ah}) and (\ref{eq:quad-lh}) we directly obtain
\begin{align}
I&=a_h(u^e,v) - a(u,v^l) + l(v^l) - l_h(v)
\\
&\lesssim h^2 \| u \|_{H^1(\Gamma)} \|v \|_\mcKh 
+ h^2 \|f \|_\Gamma\|v\|_\mcKh 
\\
&\lesssim h^2 \Big( \| u \|_{H^1(\Gamma)} 
+ \|f \|_\Gamma\Big) \|v\|_\mcKh 
\end{align}

\paragraph{Term $\bfI\bfI$.} Subtracting the quantity 
\begin{equation}
( \beta\cdot \nablas u  + \alpha u - \epsilon \Delta_\Gamma u - f)^e = 0
\end{equation}
and estimating the resulting terms we obtain
\begin{align}
II &= \tau_1 h (\nablashb u^e + \alpha_h u^e - f_h,\nablashb v)_\mcKh 
\\
&=
\tau_1h ((\nablashb u^e + \alpha_h u^e - f^e) 
-  (\beta\cdot \nablas u  + \alpha u - \epsilon \Delta_\Gamma u - f)^e ,\nablashb v)_\mcKh 
\\
 &\leq
\tau_1^{1/2} h^{1/2} \Big( \| \nablashb u^e - (\beta\cdot \nablas u)^e \|_{\mcKh} 
+ \|\alpha_h u^e - (\alpha u)^e \|_\mcKh
\\ \nonumber  
&\qquad \qquad \qquad  
+ \epsilon \| (\Delta_\Gamma u)^e \|_\mcKh 
\Big)\tau_1^{1/2}  h^{1/2} \|\nablashb v\|_\mcKh
 \\
&\leq
\tau_1^{1/2} (h^{5/2} + h^{1/2} \epsilon) \|u\|_{H^2(\Gamma)} 
\tn v \tn_h
\end{align}
Here we used the estimates
\begin{align}\nonumber
& \| \nablashb u^e - (\beta\cdot \nablas u)^e \|_{\mcKh} 
 \\
 &\qquad=
 \|\beta_h \cdot B^T \nablas u - (\beta\cdot \nablas u)^e \|_{\mcKh} 
 \\
 &\qquad \lesssim
  \|(B \beta_h) \cdot \nablas u - \beta\cdot \nablas u \|_{\Gamma}
   \\
 &\qquad \lesssim 
 \| (B \beta_h) - \beta  \|_{L^\infty(\Gamma)} 
 \|\nablas u \|_{\Gamma} 
    \\
&\qquad \lesssim h^2 \| u \|_{H^1(\Gamma)}
\end{align}
where  we used Assumption C and (\ref{eq:B-detbound}) to conclude 
that 
\begin{align}
 \| (B \beta_h) - \beta  \|_{L^\infty(\Gamma)} 
&
\lesssim 
 \| |B|^{-1} B \beta_h - \beta  \|_{L^\infty(\Gamma)} 
 + \| (|B|^{-1} - 1) B \beta_h \|_{L^\infty(\Gamma)} 
\lesssim 
h^2
\end{align}
Next again using Assumption C and  (\ref{eq:B-detbound}) 
we have
\begin{align}
 \|\alpha_h u^e - (\alpha u)^e \|_\mcKh 
 &\lesssim 
  \|(\alpha_h - \alpha^e) u^e \|_\mcKh
  \\
   &\lesssim 
  \|(\alpha_h - \alpha^e)\|_{L^\infty(\Gammah)} \| u^e \|_\mcKh
  \\
  &\lesssim h^2 \|u \|_\Gamma
\end{align}
Finally, the last term is estimated as follows
\begin{align}
\|(\Delta_\Gamma u)^e \|_\mcKh  &\lesssim 
\| \Delta_\Gamma u \|_\Gamma\lesssim \| u \|_{H^2(\Gamma)}
\end{align}
and thus 
\begin{equation}
\tau_1^{1/2} h^{1/2} \epsilon \|(\Delta_\Gamma u)^e \|_\mcKh 
\lesssim 
\tau_1^{1/2} h^{1/2} \epsilon \| u \|_{H^2(\Gamma)}
\end{equation}
%

\paragraph{Term $\bfI \bfI \bfI$.} We directly obtain
\begin{align}
III&\lesssim \tau_2h^\gamma \|n_h \cdot \nabla u^e \|_\mcTh  \| n_h \cdot \nabla v \|_\mcTh
\\
&\lesssim \tau_2h^\gamma  \|(n_h - n ) \cdot \nabla u^e \|_\mcTh  
\| n_h \cdot \nabla v \|_\mcTh
\\
&\lesssim\tau_2^{1/2}  h^{\gamma/2} \|(n_h - n )\|_{L^\infty(\Gammah)} 
\| \nabla u^e \|_\mcTh 
\tau_2^{1/2} h^{\gamma/2} \| n_h \cdot \nabla v \|_\mcTh
\\
&\lesssim \tau_2^{1/2} h^{(\gamma+3)/2 }   \| u \|_{H^1(\Gamma)} \tn v \tn_h 
\end{align}
Collecting the estimates for the terms $I$-$III$ yields the result.  
\end{proof}

Note that from the proof we see that $\gamma$ must be larger or equal to zero.
This lower bound on $\gamma$ guarantees that the stabilization is weak enought not to affect the optimal order of convergence, see Theorem \ref{thm:apriori-energy} in the next section. 

\subsection{A Priori Error Estimate}
In this section we will prove an a priori error estimates that is
optimal for both convection and diffusion dominated flows. In the
convection dominated regime the error measured in the streamline
derivative norm is optimal, $O(h)$, whereas the error in the
$L^2$-norm is suboptimal with $O(h^{1/2})$. In the diffusion dominated
regime, we show that the error in the $H^1$-norm is optimal $O(h)$. In
the latter case it is also possible to prove optimal error estimates
in the $L^2$-norm following \cite{OlReGr09,BuHaLa15}, we leave the details of this
estimate to the reader.
\begin{thm}\label{thm:apriori-energy} Let $u$ be the solution to (\ref{eq:conva}) and $u_h$ the finite element approximation 
defined by (\ref{eq:fem}). If Assumptions A-C hold and $\tau_1$ is chosen as in equation \eqref{eq:tau1}
then, there is a constant 
such that for all $h\in (0,h_0]$, with $h_0$ small enough, and $0 \leq \gamma<2$,
\begin{equation}
\tn u^e - u_h \tn_h \lesssim \Big(\max(\beta_\infty^{\frac12} h^{\frac32}, \epsilon^{\frac12} h)+ \tau_2^{1/2} h^{(\gamma+3)/2}\Big) \| u \|_{H^2(\Gamma)} 
+ h^2 \|f \|_\Gamma
\end{equation}
\end{thm}
\begin{proof} Adding and subtracting an interpolant and using the 
triangle inequality
\begin{align}\label{eq:apriori}
\tn u^e - u_h \tn_h &\leq \tn u^e - \pi_h u \tn_h 
+  \tn \pi_h u^e - u_h \tn_h
\\
&\lesssim \min(\beta_\infty^{\frac12} h^{\frac32}, \epsilon^{\frac12} h) \|u \|_{H^2(\Gamma)} 
+ \tn \pi_h u^e - u_h \tn_h
\end{align}
where we used the energy norm interpolation estimate (\ref{eq:interpol-energy}) for the first term. For the second term we obtain using coercivity, Lemma \ref{lem:coercivity} and $\gamma<2$,
\begin{equation}
\tn \pi_h u^e - u_h \tn_h \lesssim \sup_{v \in V_h \setminus \{ 0\}}
 \frac{A_h(\pi_h u^e - u_h,v)}{\tn v \tn_h}
\end{equation}
Here we have the identity 
\begin{align}
A_h(\pi_h u^e - u_h,v) 
&=
A_h (\pi_h u^e - u^e,v) + A_h (u^e - u_h,v)
\\
&=
A_h (\pi_h u^e - u^e,v) + A_h (u^e,v)  - L_h(v)
\end{align}
Setting $\eta = \pi_h u^e - u^e$ in the continuity result, in Lemma~\ref{lem:continuity}, using that $\gamma<2$, and the interpolation error estimates we get 
\begin{equation}\label{eq:cont}
A_h (\pi_h u^e - u^e,v)  \lesssim  \Big( \tau_1^{-1/2} h^{3/2}+\max(\beta_\infty^{\frac12} h^{\frac32}, \epsilon^{\frac12} h) \Big) \|u \|_{H^2(\Gamma)}  \tn v \tn_h
\end{equation}
The consistency result, Lemma~\ref{lem:consistency} yields
\begin{equation} 
A_h (u^e,v)  - L_h(v) \lesssim  \Big( h^2 +\tau_2^{1/2}h^{(\gamma+3)/2} + \tau_1^{1/2}(h^{5/2} + \epsilon h^{1/2} )\Big)  \| u  \|_{H^2(\Gamma)} \tn v \tn_h 
+h^2 \|f\|_{\Gamma} \tn v \tn_h 
\end{equation}

\paragraph{Scaling with respect to the Peclet number}
Collecting the above estimates yields the
following bound of the discrete error, if high order terms are neglected,
\[
\tn \pi_h u^e - u_h \tn_h \lesssim \Big(\tau_1^{-1/2} h^{3/2}+
\max(\beta_\infty^{\frac12} h^{\frac32}, \epsilon^{\frac12} h) + \tau_2^{1/2} h^{(\gamma+3)/2}+ \tau_1^{1/2} \epsilon h^{1/2} \Big) \|u\|_{H^2(\Gamma)}+h^2 \|f\|_{\Gamma} 
\]

From the definition of $\tau_1$ (see equation \eqref{eq:tau1}) we have two cases. If $\beta_\infty h  \ge \epsilon$ (the high Peclet number regime) then $\tau_1^{-1/2} h^{3/2}\lesssim \beta_\infty^{\frac12} h^{3/2}$,  $\epsilon^{\frac12} h \leq \beta_\infty^{\frac12} h^{3/2}$, and $\tau_1^{1/2} \epsilon h^{1/2} \lesssim \beta_\infty^{\frac12} h^{3/2}$. Thus, 
\[
\tn \pi_h u^e - u_h \tn_h \lesssim \Big(\beta_\infty^{\frac12}  h^{3/2}+ \tau_2^{1/2} h^{(\gamma+3)/2}\Big) \|u\|_{H^2(\Gamma)}+h^2 \|f\|_{\Gamma} 
\]
If on the other hand $\beta_\infty h  \leq \epsilon$ (low Peclet
number regime), then $\tau_1^{-1/2} h^{3/2} \lesssim \epsilon^{1/2} h$, $\beta_\infty^{\frac12} h^{\frac32}\leq\epsilon^{1/2} h$, and $\tau_1^{1/2} h^{1/2} \epsilon \lesssim \epsilon^{1/2} h$ and therefore
\[
\tn \pi_h u^e - u_h \tn_h \lesssim \Big(\epsilon^{1/2} h+ \tau_2^{1/2} h^{(\gamma+3)/2}\Big) \|u\|_{H^2(\Gamma)}+h^2 \|f\|_{\Gamma} 
\]

Using these estimates and equation \eqref{eq:apriori} proves the theorem.  
\end{proof}

\section{Condition Number Estimate}
Let $\{\varphi_i\}_{i=1}^N$ be the standard piecewise linear 
basis functions associated with the nodes in $\mcT_h$ and 
let $\mcA$ be the stiffness matrix with elements 
$a_{ij} = A_h(\varphi_i,\varphi_j)$. The condition number is 
defined by
\begin{equation}\label{cond_def}
\kappa_h(\mcA) := | \mcA |_{\IR^N} |\mcA^{-1} |_{\IR^N}
\end{equation}
Using the approach in \cite{BuHaLa15}, see also \cite{ErGu06}, 
we may prove the following bound on the condition number of 
the matrix.
\begin{thm}\label{thm:condition-number} 
The condition number of the stiffness matrix $\mcA$ satisfies the estimate
\begin{equation}\label{eq:condition-number}
\kappa_h(\mcA)\lesssim h^{\min(1,2-\gamma)-2}(1 + \epsilon h^{-1} )
\end{equation}
for all $h \in (0,h_0]$ with $h_0$ small enough and $0\leq \gamma < 2$. In particular, for $\gamma = 1$, we obtain the optimal estimate
\begin{equation}
\kappa_h(\mcA)\lesssim h^{-1}(1 + \epsilon h^{-1} )
\end{equation}
\end{thm} 
\begin{proof}
First we note that if $v = \sum_{i=1}^N V_{i} \varphi_i$ 
and $\{\varphi_i\}_{i=1}^N$ is the usual nodal basis on 
$\mcT_h$ then the following well known estimates hold
\begin{equation}\label{rneqv}
c h^{-d/2} \| v \|_{\mcT_h} \leq | V |_{\IR^N} \leq C h^{-d/2}\| v \|_{\mcT_h}
\end{equation}
It follows from the definition (\ref{cond_def}) of the condition 
number that we need to estimate $| \mcA |_{\IR^N}$ and 
$|\mcA^{-1}|_{\IR^N}$. 

\paragraph{Estimate of $| \mcA |_{\IR^N}$.} We have
\begin{align}
|\mcA V|_{\IR^N} &= \sup_{W \in \IR^N \setminus 0} \frac{(W,\mcA V)_{\IR^N}}{| W |_{\IR^N}}
\\
&= \sup_{w \in V_h \setminus 0 }  \frac{A_h(v,w)}{| W |_{\IR^N}}
\\
&\lesssim h^{d-2}(1+\epsilon h^{-1}) | V |_{\IR^N}
\end{align}
where we used the continuity 
\begin{equation}\label{eq:cond-Ah-cont}
A_h(v,w) \lesssim h^{d-2}(1+\epsilon h^{-1}) |V|_{\IR^N} | W |_{\IR^N}
\end{equation}
To verify (\ref{eq:cond-Ah-cont}) we use inverse estimates to 
derive bounds in terms of $\|v\|_{\mcTh}$ and $\|w\|_{\mcTh}$ 
and then we employ (\ref{rneqv}) to pass over to the $|\cdot |_\IR^N$ norms
\begin{align}
a_h(v,w) &\lesssim 
\| \beta_h \cdot \nablash v \|_{\mcK_h} \| w \|_{\mcK_h} 
+ \| \alpha_h v \|_{\mcK_h}
\|  w \|_{\mcK_h}
+ \epsilon \| \nablash v \|_\mcKh \| \nablash w \|_\mcKh 
\\
&\lesssim 
h^{-1} \|  \nabla v \|_{\mcT_h} \| w \|_{\mcT_h} 
+ h^{-1} \| v \|_{\mcT_h}\|  w \|_{\mcT_h}
+ \epsilon h^{-1} \| \nabla v \|_\mcTh \| \nabla w \|_\mcTh 
\\
&\lesssim 
(h^{-2} 
+ h^{-1} 
+ \epsilon h^{-3}) \| v \|_\mcTh \|  w \|_\mcTh 
\\
&\lesssim h^{d-2}(1+\epsilon h^{-1}) |V|_{\IR^N} | W |_{\IR^N}
\\ \nonumber
\\
s_{h,1}(v,w) &= \tau_1 h (\nablashb v + \alpha_h v, \nablashb w)_\mcKh
\\
&\lesssim h\|\nablashb v \|_\mcKh \|\nablashb w \|_\mcKh 
+ h \|\alpha_h v \|_\mcKh \|\nablashb w \|_\mcKh
\\
&\lesssim \| \nabla v \|_\mcTh \| \nabla w \|_\mcTh 
+  \|v \|_\mcTh \| \nabla w \|_\mcTh
\\
&\lesssim ( h^{-2} + h^{-1} )\| v \|_\mcTh \| w \|_\mcTh 
\\
&\lesssim 
h^{d-2} |V|_{\IR^N} | W |_{\IR^N}
\\ \nonumber
\\
s_{h,2}(v,w) &\lesssim h^\gamma \|n_h \cdot \nabla v \|_\mcTh 
\|n_h \cdot \nabla w \|_\mcTh
\\
 &\lesssim h^\gamma \| \nabla v \|_\mcTh 
\| \nabla w \|_\mcTh
\\
 &\lesssim h^{\gamma-2} \| v \|_\mcTh 
\| w \|_\mcTh
\\
&\lesssim 
h^{d+\gamma-2} |V|_{\IR^N} | W |_{\IR^N}
\\
&\lesssim 
h^{d-2} |V|_{\IR^N} | W |_{\IR^N}
\end{align}
where we used the inverse estimate 
$\| w \|_{K} \lesssim \textcolor{red}{h^{-\frac12}} \| w \|_T$, where $K = T \cap \Gammah$  
to pass from $\mcK_h$ to $\mcT_h$, and the standard 
inverse estimate $\| \nabla w \|_T \lesssim h^{-1} \| w \|_T$ to 
remove the gradient.

We conclude that
\begin{equation}\label{Abound}
| \mcA |_{\IR^N} \lesssim h^{d-2}(1+\epsilon h^{-1})
\end{equation}

\paragraph{Estimate of $|\mcA^{-1} |_{\IR^N}$.} We note that 
using \eqref{rneqv} and Lemma \ref{lem:poincare-normalstab} 
we have 
\begin{equation}\label{eq:condtechnicala}
 h^d | V |^2_{\IR^N}  
 \lesssim \|v\|^2_{\mcT_h} 
 \lesssim h \|v\|_{\mcK_h}^2  + h^{2-\gamma} \| v \|^2_{s_h} 
 \lesssim h^{\min(1,2-\gamma)} \tn v \tn_h^2
\end{equation}
Thus it follows that 
\begin{equation}\label{eq:condtechnicalaa}
 h^{d-\min(1,2-\gamma)} | V |^2_{\IR^N} 
 =   
 h^{\widetilde{d}} | V |^2_{\IR^N}
 \lesssim 
 \tn v \tn_h^2
\end{equation}
where we introduced the notation $\widetilde{d} = d-\min(1,2-\gamma)$. Starting from \eqref{eq:condtechnicala} and using the coercivity 
(\ref{eq:coercivity})  we obtain 
\begin{align}
 | V |_{\IR^N} 
 \lesssim h^{-\widetilde{d}/2} \tn
 v \tn_h 
&\lesssim h^{-\widetilde{d}/2} \sup_{w \in V_h \setminus \{0\} }  \frac{A_h(v,w)}{\tn w
  \tn_h} 
\\  
&\qquad  
\lesssim \sup_{W \in \IR^N\setminus \{0\}}
 h^{-\widetilde{d}/2} \frac{ | \mcA V|_{\IR^N}  |W|_{\IR^N}
}{h^{{\widetilde{d}}/{2}} |W|_{\IR^N}} \lesssim h^{-\widetilde{d}}  | \mcA V|_{\IR^N} 
\end{align}
where we used (\ref{eq:condtechnicala}), 
$h^{\widetilde{d}/2} |W|_{\IR^N} \lesssim \tn w \tn_h$,
to replace $\tn w \tn_h$ by $h^{d/2} |W|_{\IR^N} $  in the denominator.
Setting $V= \mcA^{-1} X$, $X \in \IR^N$, we obtain
\begin{equation}\label{Ainvbound}
| \mcA^{-1}  |_{\IR^N} \lesssim h^{-\widetilde{d}} =h^{-(d-\min(1,2-\gamma))}  
\end{equation}

\paragraph{Conclusion.} Combining the estimates \eqref{Abound} and \eqref{Ainvbound} the estimate (\ref{eq:condition-number})  follows.
\end{proof}

\section{Numerical Examples}

\subsection{Convection--Diffusion} 
We consider convection--diffusion on the spheroid defined by 
\[
\frac{(x-1/2)^2 + (y-1/2)^2}{r_\text{max}^2} + \frac{(z-1/2)^2}{r_\text{min}^2} = 1
\]
with $r_\text{max}=1/2$ and $r_\text{min}=1/4$. The convective velocity was chosen as
\[
\beta=(1/2-y,x-1/2,0) 
\]
and parameters $\alpha=0$, $c_\tau=1/2$ in (\ref{eq:tau1}), $\gamma=0$ in (\ref{eq:sh-two}). The right-hand $f$ is set by 
applying the differential operator to the fabricated solution
\[
u(x,y,z) = 100(x-1/2)(y-1/2)(z-1/2)
\]
In Fig \ref{fig:iso} we show an isoplot of the solution using $\epsilon = 10^{-3}$ on a given mesh in a sequence of refinements,
and in Fig. \ref{fig:velo} we show the velocity field plotted on the same mesh.
Finally, in Fig. \ref{fig:conv} we present the convergence in $L_2(\Gamma_h)$ obtained by our method, close to second order.

\subsection{Convection--Reaction with a Layer} 

We consider convection--diffusion on the spheroid defined by 
\[
\frac{(x-1/2)^2 + (y-1/2)^2}{r_\text{max}^2} + \frac{(z-1/2)^2}{r_\text{min}^2} = 1
\]
with $r_\text{max}=0.5$ and $r_\text{min}=0.45$. The convective velocity was chosen as
\[
\beta= (5-10 y,10 x-5,0) 
\]
and parameters $\alpha=1$, $\epsilon = 0$. The right-hand $f$ was chosen as
\[
f= \left\{\begin{array}{c}
\text{1 if $z > 0.55$}\\
\text{0 if $z \leq 0.55$}\end{array}\right.
\]
creating a discontinuity at $z=0.55$.
In Fig \ref{fig:iso2} we show isoplots of the solution using $c_\tau=0$, $\gamma = 10^{-4}$ (top), and $\gamma = 10^{3}$ (bottom)on a given mesh.
notice the instability for small $\gamma$ end and excessive diffusivity for large $\gamma$.
In Fig. \ref{fig:iso3} we show the corresponding isoplot for $\gamma=1$, $c_\tau=0$, and in Fig. \ref{fig:iso4} we used $c_\tau=1/2$ $\gamma=0$.
In both cases there are, as expected, slight over- and undershoots close to the discontinuity.

\section*{Acknowledgements}
This research was supported in part by the Swedish Foundation for Strategic Research Grant No.\ AM13-0029 (PH,MGL), the Swedish Research Council Grants Nos.\ 2011-4992 (PH) and 2013-4708 (MGL), and EPSRC, UK, 
Grant Nr. EP/P01576X/1. (EB)

  \newpage
\begin{figure}
\begin{center}
\includegraphics[scale=0.3]{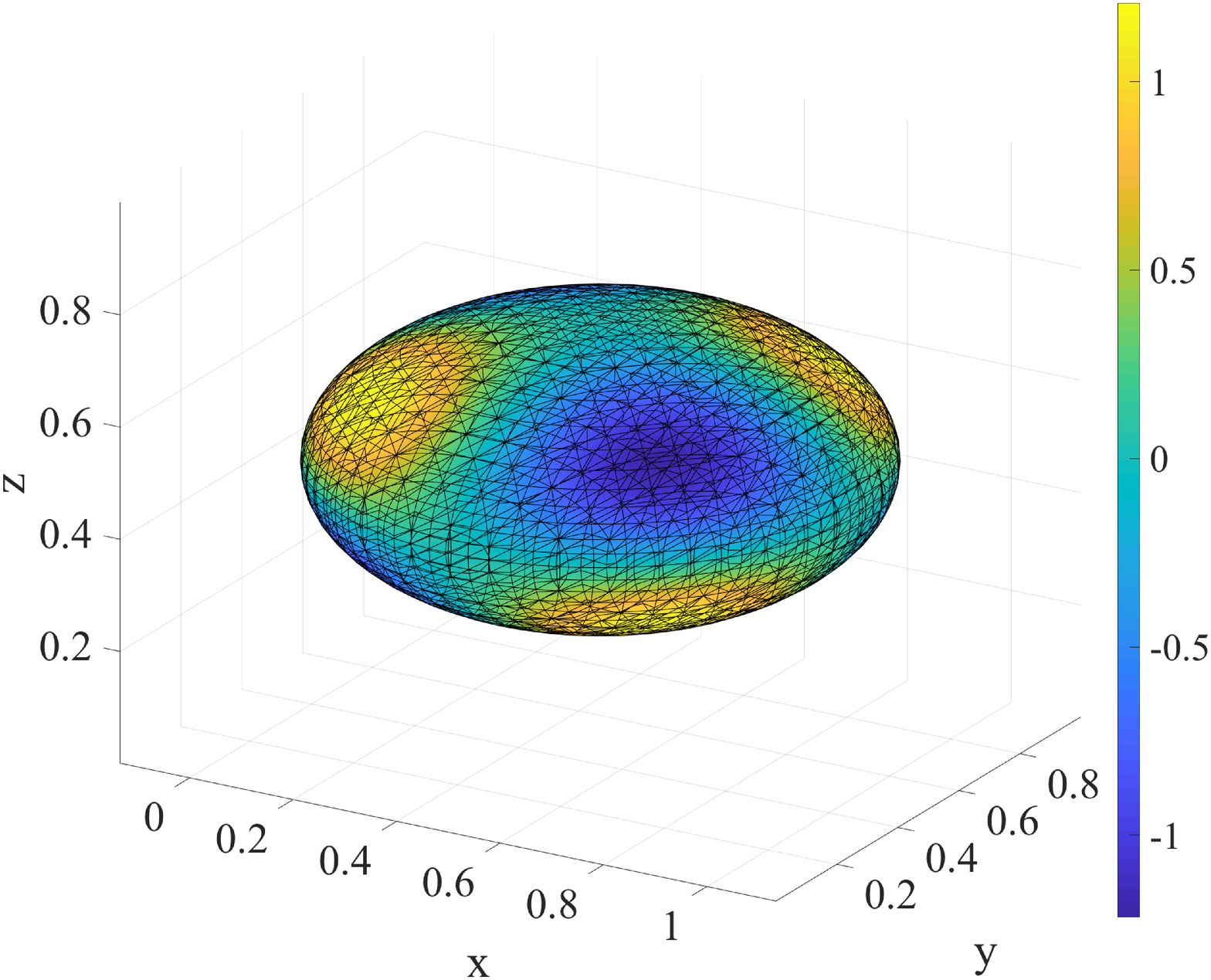}
\end{center}
\caption{Isoplot of the solution on a given mesh.\label{fig:iso}}
\end{figure}
\begin{figure}
\begin{center}
\includegraphics[scale=0.3]{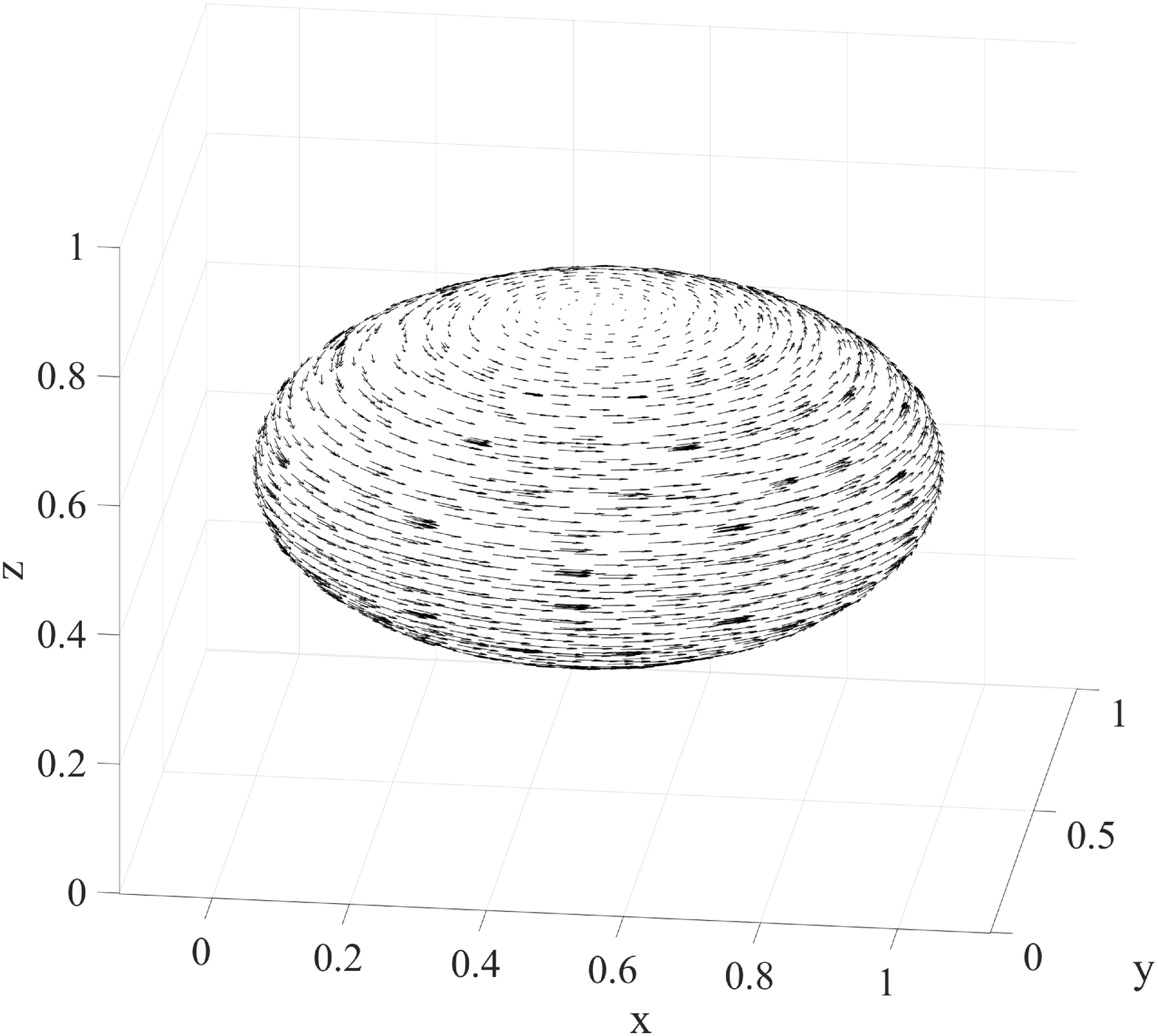}
\end{center}
\caption{Velocity field on a given mesh.\label{fig:velo}}
\end{figure}
\begin{figure}
\begin{center}
\includegraphics[scale=0.3]{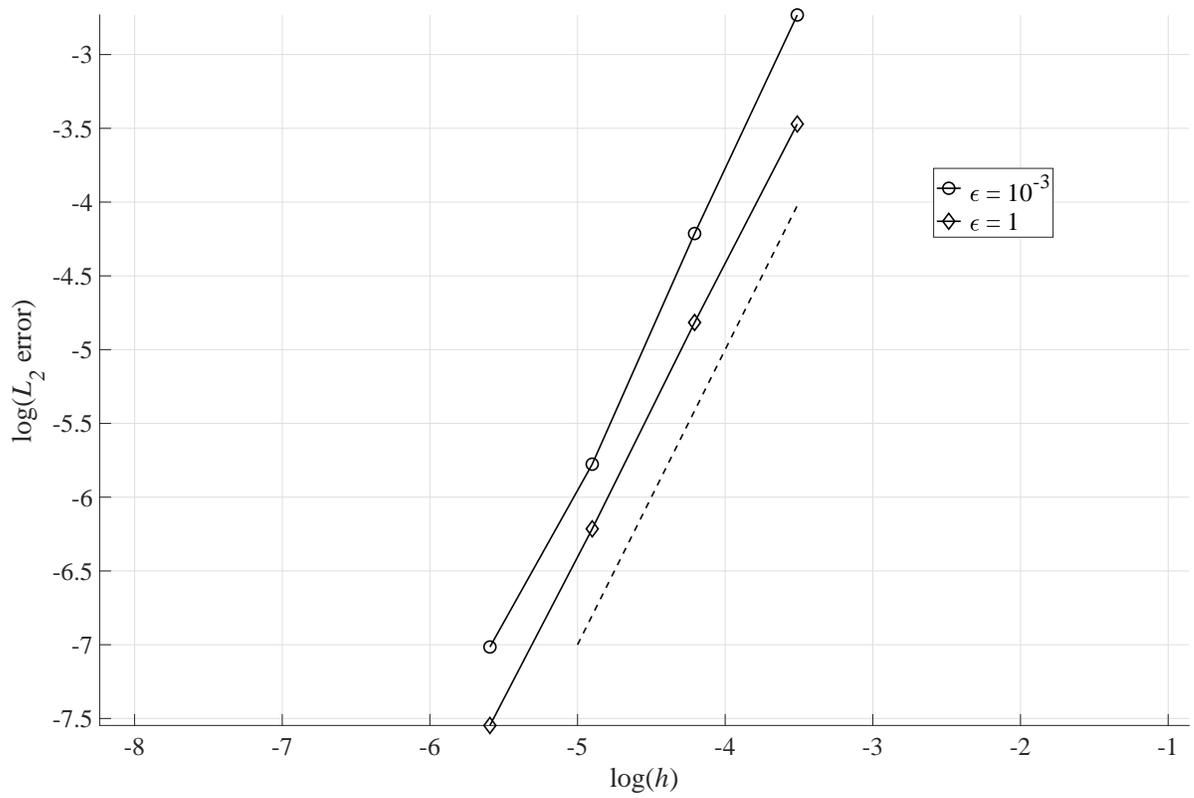}
\end{center}
\caption{$L_2(\Gamma_h)$--convergence of the discrete solution. Dotted line indicates second order convergence.\label{fig:conv}}
\end{figure}
\begin{figure}
\begin{center}
\includegraphics[scale=0.3]{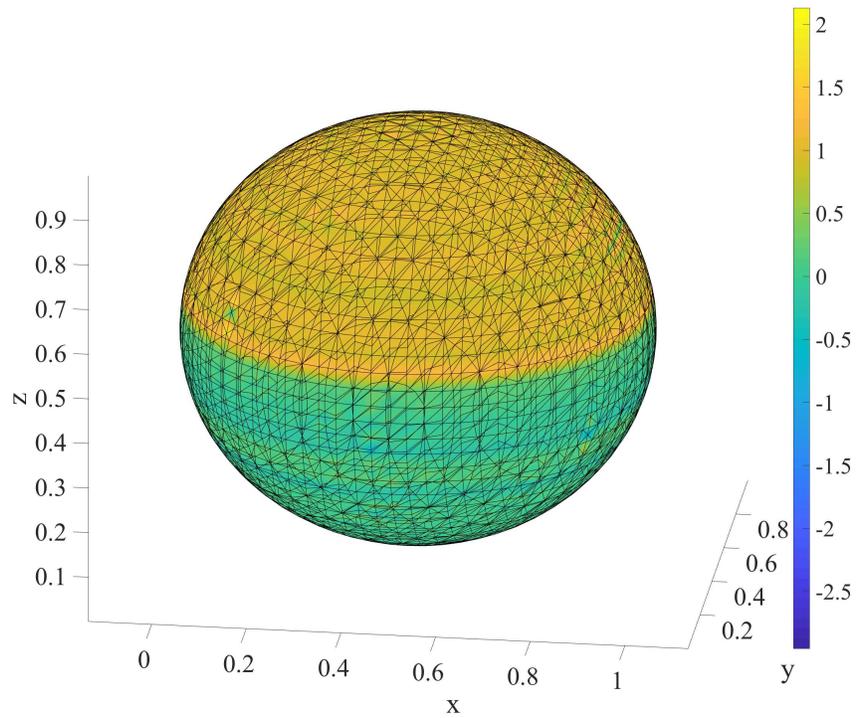}
\includegraphics[scale=0.3]{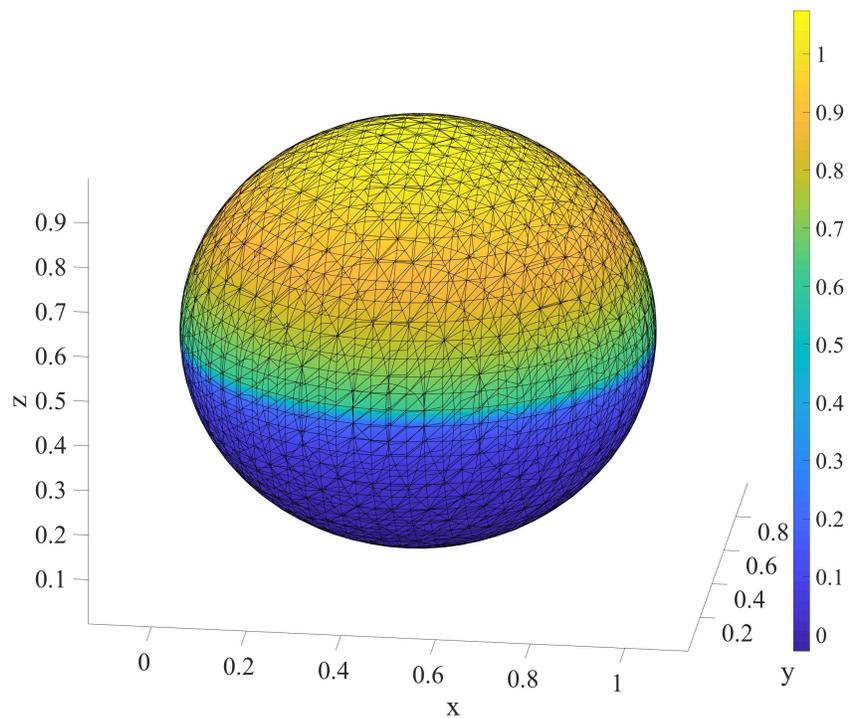}
\end{center}
\caption{Isoplot of the solution on a given mesh for $c_\tau=0$. Top: $\gamma = 10^{-4}$, bottom: $\gamma = 10^{3}$\label{fig:iso2}}
\end{figure}
\begin{figure}
\begin{center}
\includegraphics[scale=0.3]{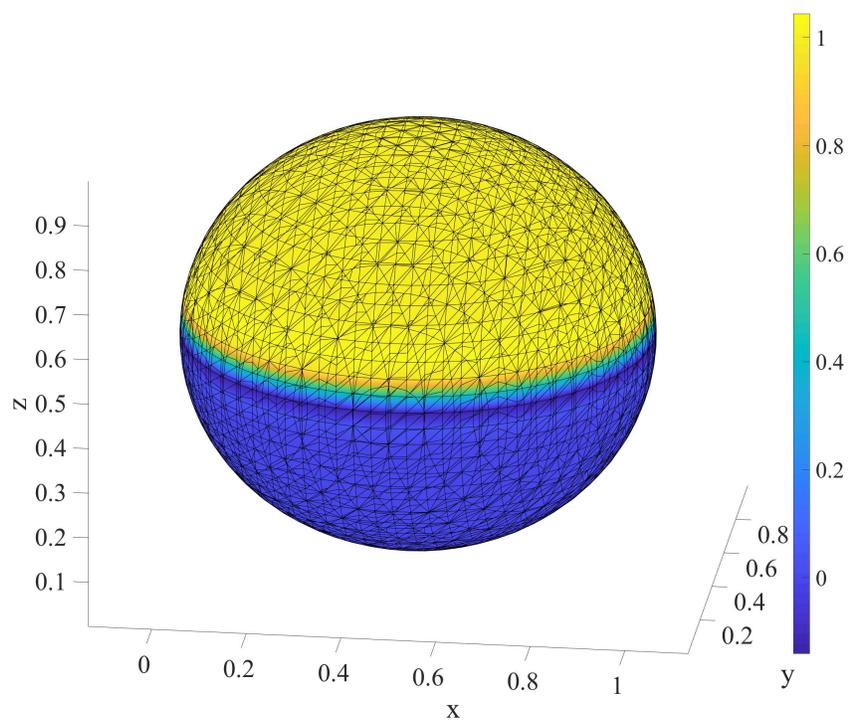}
\end{center}
\caption{Isoplot of the solution for $c_\tau=0$, $\gamma = 1$\label{fig:iso3}}
\end{figure}
\begin{figure}
\begin{center}
\includegraphics[scale=0.3]{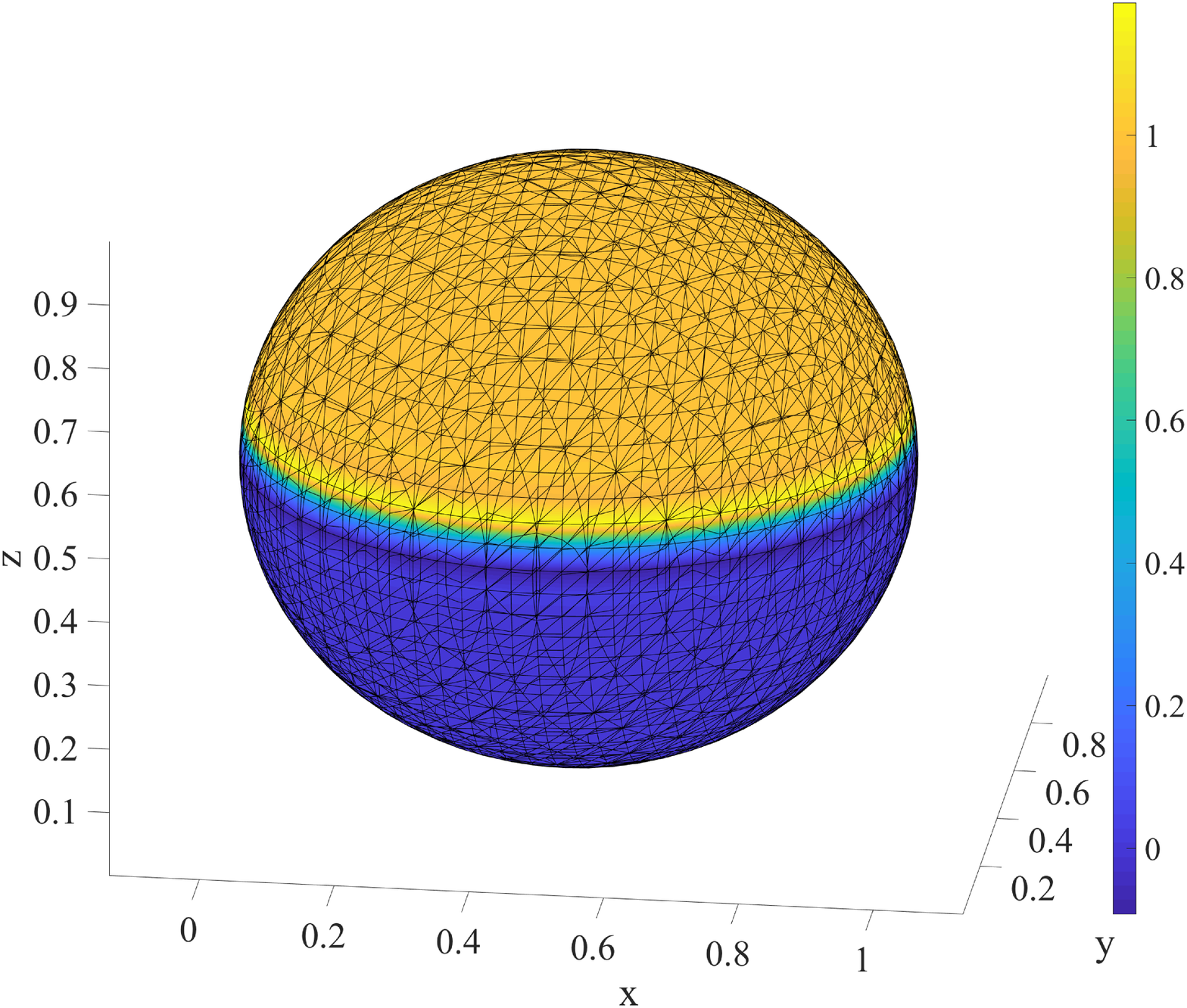}
\end{center}
\caption{Isoplot of the solution for $\gamma = 0$, $c_\tau=1/2$\label{fig:iso4}}
\end{figure}


\begin{thebibliography}{10}

\bibitem{BrHu82}
A.~N. Brooks and T.~J.~R. Hughes.
\newblock Streamline upwind/{P}etrov-{G}alerkin formulations for convection
  dominated flows with particular emphasis on the incompressible
  {N}avier-{S}tokes equations.
\newblock {\em Comput. Methods Appl. Mech. Engrg.}, 32(1-3):199--259, 1982.
\newblock FENOMECH '81, Part I (Stuttgart, 1981).

\bibitem{BuHaLa15}
E.~Burman, P.~Hansbo, and M.~G. Larson.
\newblock A stabilized cut finite element method for partial differential
  equations on surfaces: the {L}aplace-{B}eltrami operator.
\newblock {\em Comput. Methods Appl. Mech. Engrg.}, 285:188--207, 2015.

\bibitem{BuHaLaMa16}
E.~Burman, P.~Hansbo, M.~G. Larson, and A.~Massing.
\newblock Cut finite element methods for partial differential equations on
  embedded manifolds of arbitrary codimensions.
\newblock Technical report, Mathematics, Ume{\aa} University, Sweden, 2016.
\newblock arXiv:1610.01660.

\bibitem{BuHaLaMaZa16}
E.~Burman, P.~Hansbo, M.~G. Larson, A.~Massing, and S.~Zahedi.
\newblock Full gradient stabilized cut finite element methods for surface
  partial differential equations.
\newblock {\em Comput. Methods Appl. Mech. Engrg.}, 310:278--296, 2016.

\bibitem{BuHaLaZa15}
E.~Burman, P.~Hansbo, M.~G. Larson, and S.~Zahedi.
\newblock Stabilized cutfem for the convection problem on surfaces.
\newblock Technical report, Mathematics, Ume{\aa} University, Sweden, 2015.
\newblock arXiv:1511.02340.

\bibitem{BuHaLaZa16}
Erik Burman, Peter Hansbo, Mats~G. Larson, and Sara Zahedi.
\newblock Cut finite element methods for coupled bulk-surface problems.
\newblock {\em Numer. Math.}, 133(2):203--231, 2016.

\bibitem{De09}
A.~Demlow.
\newblock Higher-order finite element methods and pointwise error estimates for
  elliptic problems on surfaces.
\newblock {\em SIAM J. Numer. Anal.}, 47(2):805--827, 2009.

\bibitem{Dz88}
G.~Dziuk.
\newblock Finite elements for the {B}eltrami operator on arbitrary surfaces.
\newblock In {\em Partial differential equations and calculus of variations},
  volume 1357 of {\em Lecture Notes in Math.}, pages 142--155. Springer,
  Berlin, 1988.

\bibitem{DzEl13}
G.~Dziuk and C.~M. Elliott.
\newblock Finite element methods for surface {PDE}s.
\newblock {\em Acta Numer.}, 22:289--396, 2013.

\bibitem{ErGu06}
A.~Ern and J.-L. Guermond.
\newblock Evaluation of the condition number in linear systems arising in
  finite element approximations.
\newblock {\em M2AN Math. Model. Numer. Anal.}, 40(1):29--48, 2006.

\bibitem{GiTr83}
D.~Gilbarg and N.~S. Trudinger.
\newblock {\em Elliptic partial differential equations of second order}.
\newblock Classics in Mathematics. Springer-Verlag, Berlin, 2001.
\newblock Reprint of the 1998 edition.

\bibitem{GLR18}
J\"org Grande, Christoph Lehrenfeld, and Arnold Reusken.
\newblock Analysis of a high-order trace finite element method for {PDE}s on
  level set surfaces.
\newblock {\em SIAM J. Numer. Anal.}, 56(1):228--255, 2018.

\bibitem{GrOlRe15}
S.~Gross, M.~A. Olshanskii, and A.~Reusken.
\newblock A trace finite element method for a class of coupled bulk-interface
  transport problems.
\newblock {\em ESAIM Math. Model. Numer. Anal.}, 49(5):1303--1330, 2015.

\bibitem{HaLaLa16}
P.~Hansbo, M.~G. Larson, and K.~Larsson.
\newblock Analysis of finite element methods for vector {L}aplacians surfaces.
\newblock Technical report, Mathematics, Ume{\aa} University, Sweden, 2016.
\newblock arXiv:1610.06747.

\bibitem{HaLaZa15}
P.~Hansbo, M.~G. Larson, and S.~Zahedi.
\newblock Characteristic cut finite element methods for convection--diffusion
  problems on time dependent surfaces.
\newblock {\em Comput. Methods Appl. Mech. Engrg.}, 293:431--461, 2015.

\bibitem{JoNaPi84}
Claes Johnson, Uno N{\"a}vert, and Juhani Pitk{\"a}ranta.
\newblock Finite element methods for linear hyperbolic problems.
\newblock {\em Comput. Methods Appl. Mech. Engrg.}, 45(1-3):285--312, 1984.

\bibitem{OlRe14}
M.~A. Olshanskii and A.~Reusken.
\newblock Error analysis of a space-time finite element method for solving
  {PDE}s on evolving surfaces.
\newblock {\em SIAM J. Numer. Anal.}, 52(4):2092--2120, 2014.

\bibitem{OlReGr09}
M.~A. Olshanskii, A.~Reusken, and J.~Grande.
\newblock A finite element method for elliptic equations on surfaces.
\newblock {\em SIAM J. Numer. Anal.}, 47(5):3339--3358, 2009.

\bibitem{OlReXu14a}
M.~A. Olshanskii, A.~Reusken, and X.~Xu.
\newblock An {E}ulerian space-time finite element method for diffusion problems
  on evolving surfaces.
\newblock {\em SIAM J. Numer. Anal.}, 52(3):1354--1377, 2014.

\bibitem{OlReXu14b}
M.~A. Olshanskii, A.~Reusken, and X.~Xu.
\newblock A stabilized finite element method for advection-diffusion equations
  on surfaces.
\newblock {\em IMA J. Numer. Anal.}, 34(2):732--758, 2014.

\bibitem{Reu15}
A.~Reusken.
\newblock Analysis of trace finite element methods for surface partial
  differential equations.
\newblock {\em IMA J. Numer. Anal.}, 35(4):1568--1590, 2015.

\end{thebibliography}
\end{document}